\documentclass[11pt,reqno]{amsart}

\usepackage{amsmath,amsthm,amssymb,amsfonts,mathrsfs,mathtools}
\usepackage{enumitem}
\usepackage{tikz-cd}
\usepackage{hyperref}
\usepackage{graphicx}
\usepackage{stmaryrd} % For ⟦ ⟧ brackets
\usepackage{xcolor}
\usepackage[cmtip,all]{xy}

\theoremstyle{plain}
\newtheorem{theorem}{Theorem}[section]

\newtheorem{cor}[theorem]{Corollary}
\newtheorem{prop}[theorem]{Proposition}

\theoremstyle{definition}
\newtheorem{defn}[theorem]{Definition}
\newtheorem{example}[theorem]{Example}
\newtheorem{notation}[theorem]{Notation}
\theoremstyle{remark}
\newtheorem{remark}[theorem]{Remark}

\numberwithin{equation}{section}

\newcommand{\RR}{\mathbb{R}}

\newcommand{\Z}{\mathbb{Z}}
\DeclareMathOperator{\Hom}{Hom}

\DeclareMathOperator{\Burn}{Burn}
\DeclareMathOperator{\Sp}{Sp}

\newcommand{\uM}{\underline{M}}

\newcommand{\uZ}{\underline{\mathbb{Z}}}

\newcommand{\F}{\mathbb{F}_2}
\newcommand{\uFz}{\underline{\mathbb{F}_2}}

\newcommand{\Rep}{\mathrm{Rep}}

\newcommand{\bs}{\bigstar}
\newcommand{\tH}{\tilde{H}}
\newcommand{\res}{\mathrm{res}}

\usepackage{tikz}
\usepackage{spectralsequences}

\title{Bredon cohomology methods in mass partition problems on spheres}
\author{}
\date{}

\begin{document}
\author{Surojit Ghosh}
\address{Department of Mathematics, Indian Institute of Technology, Roorkee, Uttarakhand-247667, India}
\email{surojit.ghosh@ma.iitr.ac.in; surojitghosh89@gmail.com}
\date{\today}

\begin{abstract}
We apply $\mathrm{RO}(G)$-graded Bredon cohomology to mass assignment problems, extending classical mass partition methods. Within this framework, we reprove a recent result of Lessure and Soberón: for \(n+1\) mass assignments on \(k\)-dimensional affine subspaces of \(\mathbb{R}^n\), there exists a \(k\)-subspace containing a sphere that simultaneously bisects all measures. This approach highlights a flexible topological framework with potential for broader applications. 
\end{abstract}

\maketitle

\section{Introduction}

Mass-partition problems lie at the crossroads of combinatorial geometry, measure theory, and algebraic topology. 
The classical \emph{ham-sandwich theorem} \cite{Ste38, BZ04} asserts that any $d$ finite Borel measures on $\mathbb{R}^d$ can be simultaneously bisected by a single hyperplane, a result proved by Stone and Tukey \cite{ST42} via the Borsuk--Ulam theorem. 
This paradigm, translating topological symmetry arguments into fair-division results, has inspired a vast literature.  

Subsequent generalizations include the Grünbaum--Hadwiger--Ramos problem on multiple hyperplanes \cite{Gru60, Had66, Ram96}, the necklace-splitting theorem \cite{Alo87}, and various extensions involving spheres, wedges, or fans \cite{RS20}. 
A particularly flexible framework is provided by \emph{mass assignments}, introduced by Mani-Levitska, Vrećica, and Živaljević \cite{MVZ06}, where measures vary continuously with affine subspaces. 
These problems are typically analyzed via the \emph{configuration-space/test-map scheme}, which encodes admissible partitions into a $G$-space $X$ 
and searches for $G$-equivariant maps $X \to S(V)$.  
Nonexistence of such maps yields partition results, with obstructions detected via equivariant cohomology and index theories \cite{FH88, MVZ06, BFHZ16, BFHZ18, AFS22}.   

Our approach applies $\mathrm{RO}(G)$-graded Bredon cohomology for elementary abelian $2$-groups $G$, as developed by Hausmann and Schwede \cite{HS24}. 
In particular, we analyze the induced map of $\tilde{H}_G^\bs(S^0; \uFz)$–modules arising from the test map $X \to S(V)$, where the action of Euler classes plays a decisive role in detecting obstructions.  

We show that these tools can be brought to bear on the problem of bisecting measures by spheres in affine subspaces.  

Concretely, we recover the following recent result of Lessure and Soberón \cite{LS25}: {\it for \(n+1\) mass assignments on \(\ell\)-dimensional affine subspaces of $\mathbb{R}^n$, there exists a \(\ell\)-subspace containing a sphere that simultaneously bisects all the measures. } Intuitively, a mass assignment on the 
$\ell$-dimensional affine subspaces of $\RR^n$
 is a rule that assigns to each such subspace a measure in a way that varies continuously with the subspace. A precise formulation, expressed in terms of sections of an appropriate fibre bundle over the space of all $\ell$-dimensional affine subspaces, can be found in \cite{AFS22}.
While the theorem itself is not new, our approach demonstrates how representation-graded methods naturally encode the required obstructions and suggest a broader applicability to unresolved mass assignment problems.

Beyond recovering the result of Lessure and Soberón, the main contribution of this paper is methodological: 
we show that $\mathrm{RO}(G)$-graded Bredon cohomology provides a flexible and conceptually unified framework for analyzing mass assignment problems.

\begin{notation}\label{not}
 \begin{enumerate}
\item For integers \(n \geq \ell \geq 1\), we denote by \(V_{n,\ell}\) the real Stiefel manifold consisting of ordered \(\ell\)-tuples of orthonormal vectors in \(\mathbb{R}^{n}\). Explicitly,
\(
   V_{n,\ell} \;=\; \Bigl\{ (v_1,\dots,v_\ell)\in (\mathbb{R}^{n})^{\ell} \,\Bigm|\, 
      \langle v_{i},v_{j}\rangle = \delta_{ij} \Bigr\}.
\)
It is a smooth compact manifold of dimension 
\(
   \dim V_{n,\ell} = n\ell - \tfrac{\ell(\ell+1)}{2}.
\)

\item 
We denote by \(\mathrm{sgn}\colon C_{2}\to\{\pm 1\}\) the sign representation.  
For each \(\alpha=(\alpha_{1},\dots,\alpha_{\ell})\in (C_{2})^{\ell} \), define the group homomorphism
\(
   f_{\alpha}\in \Hom ((C_{2})^{\ell}, C_{2}),\mbox{ by }
   f_{\alpha}(g_{1},\dots,g_{\ell})= 
   \alpha_{1}g_{1}+\dots+\alpha_{\ell}g_{\ell}.
\)
The associated one-dimensional real representation is given by
\[
   \chi_{\alpha} := \mathrm{sgn}\circ f_{\alpha},
\qquad
   \chi_{\alpha}(g_{1},\dots,g_{\ell})
      = (-1)^{\alpha_{1}g_{1}+\dots+\alpha_{\ell}g_{\ell}}.
\]
Every real irreducible representation of \((C_{2})^{\ell}\) is of this form.  
In particular, for the standard basis vectors \(e_{i}\in (C_{2})^{\ell}\) we obtain the characters
\(
   \chi_{i} := \chi_{e_{i}}, \quad 1\le i\le \ell,
\)
corresponding to the canonical coordinate projections.

\item We denote by \(\bs\) the \(RO(G)\)-grading, and by \(\ast\) the ordinary integer grading.
\end{enumerate}
\end{notation}

\paragraph{\textbf{Acknowledgement.}} The author would like to thank the referee for their detailed and pertinent comments, which have helped improve the clarity and presentation of this manuscript. 

\section{Preliminaries on Bredon Cohomology}

Ordinary cohomology theories are defined on abelian groups and represented by spectra whose homotopy groups concentrate in degree zero. In the equivariant setting, this role is played by \emph{Mackey functors}. We briefly recall their definition and relation to equivariant cohomology (see \cite{May96} for details).

The \emph{Burnside category} \(\Burn_G\) is the category whose objects are finite \(G\)-sets. For objects \(S, T\), the morphism set is the group completion of spans between \(S\) and \(T\) in the category of finite \(G\)-sets.

\begin{defn}
A \emph{Mackey functor} is an additive functor 
\[
\uM: \Burn_G^{op} \to \mathsf{Ab},
\]
from the opposite Burnside category to abelian groups.
\end{defn}

In this paper, we focus on the group \((C_2)^\ell,\) the elementary abelian 2-group of rank \(\ell\). 
\begin{example}
    For a $G$-module $M$, we define a Mackey functor $\uM$ by the formula $\uM(G/H) = M^H$. In particular, we consider the Mackey functors $\uZ$ and $\uFz$, corresponding to the trivial $G$-modules $\Z$ and $\F$, respectively.
\end{example}

For a real $G$-representation $V$ equipped with a $G$-invariant inner product, we define 
\[
S(V) := \{ v \in V \mid \langle v, v \rangle = 1 \},
\]
the unit sphere in $V$, and $S^V$, the one-point compactification of $V$. In equivariant stable  homotopy theory, the $V$-fold suspension map 
\[
X \mapsto S^V \wedge X
\]
is invertible. We refer to \cite{MM02} for constructing the equivariant stable homotopy category, denoted by $\Sp^G$, which is the homotopy category of $G$-spectra. 

Since the $V$-fold suspension map is invertible, one can define $S^\alpha$ for $\alpha \in RO(G)$, the Grothendieck group completion of the monoid of irreducible representations of $G$. This construction induces an $RO(G)$-grading on homotopy groups. Specifically, for a based $G$-space $X$, the ordinary Bredon cohomology with coefficients in the Mackey functor $\uM$ at grading $\alpha \in RO(G)$ is given by:  
\[
\tilde{H}^\alpha_G(X; \uM) := {\Sp}^G(X, S^\alpha \wedge H\uM),
\]
here $H\uM$ denotes the  Eilenberg--Mac Lane spectrum (cf. \cite{HHR16}) for the  Mackey functor $\uM$ defined as  
\[
\underline{\pi_n}(H\uM) = 
\begin{cases} 
\uM, & n = 0, \\
0, & n \neq 0.
\end{cases}
\]

Let $V$ be a $G$-representation with $V^G=\{0\}$. We denote by $a_V$ the inclusion of the fixed points,  
\[
a_V: S^0 \to S^V.
\]
For a ring spectrum $X$ with a $G$-action, we abuse notation and also denote by $a_V$ its image under the map $S^0 \to X$.

If the representation $V$ contains the trivial representation as a summand, then we set $a_V = 0$.  
Moreover, for any two $G$-representations $V$ and $W$, we have the relation  
\[
a_V a_W = a_{V \oplus W}.
\]

\noindent See \cite[Definition 3.11]{HHR16} for further details.

For an orientable $G$-representation $V: G \to SO(V)$, a choice of orientation induces an isomorphism  
\[
\tilde{H}^{\dim(V)}_G(S^V; \uZ) \cong \Z.
\]
The Thom space of the equivariant bundle $V \to G/G$ is $S^V$. In particular, the restriction map  
\[
\tilde{H}^{\dim(V)}_G(S^V; \uZ) \to \tilde{H}^{\dim(V)}_e(S^{\dim(V)}; \Z)
\]
is an isomorphism.

Utilizing the above isomorphism, for an orientable $G$-representation $V$, we define the orientation class $u_V$ as  
\[
u_V \in \tilde{H}^{V-\dim(V)}_G(S^0 ; \uZ),
\]
the generator that maps to $1$ under the restriction isomorphism. The orientation class satisfies the relations:  
\[
u_{V \oplus 1} = u_V, \quad u_V \cdot u_W = u_{V \oplus W}.
\]
\noindent See \cite[Definition 3.12]{HHR16} for more details. 
It will be convenient to use the notation
\(G^\circ := \mathrm{Hom}(G, C_2) \setminus \{1\}\)
for the set of nontrivial characters of \(G\).

\section{Equivariant formulation}

\subsection{Mass partitions as equivariant maps} 

The mass partition problem asks whether measures on \(\mathbb{R}^{n}\) can be simultaneously bisected by geometric objects such as hyperplanes or spheres.
A key insight in \cite{MVZ06}, is that such problems can be translated into the existence or nonexistence of equivariant maps between representation spheres and certain spaces.

We now describe the equivariant framework relevant for our situation.  Consider the action of \((C_{2})^{\ell}\) on the real Stiefel manifold \(V_{n,\ell}\), given by
\[
   (g_{1},\dots,g_{\ell})\cdot(v_{1},\dots,v_{\ell})
   \;=\; (g_{1}v_{1},\dots,g_{\ell}v_{\ell}),
\]
where each \(g_{i}\in \{\pm 1\}\) acts on \(\mathbb{R}^{n}\) by scalar multiplication.

Assume, for the sake of contradiction, that the following statement fails:

\begin{quote}
\emph{Given \(n+1\) mass assignments on \(\ell\)-dimensional affine subspaces of \(\mathbb{R}^{n}\), 
there exists a \(\ell\)-dimensional subspace and a sphere contained in it that simultaneously bisects all the measures.}
\end{quote}

Following \cite[page 7]{LS25}, this failure is equivalent to the existence of a continuous \((C_{2})^{\ell+1}\)-equivariant map
\[
   S^{n}\times V_{n, \ell} \;\longrightarrow\; S(V),
\]
with
\[
   V \;=\; \mathbb{R}^{n} \oplus \mathbb{R}^{n-1}  \oplus \cdots \oplus \mathbb{R}^{n-\ell}.
\]
Where the action of \((C_2)^{\ell+1}\) on \(S^{n}\times V_{n, \ell}\) as the product
of the antipodal action of \(C_2\) on \(S^n\) and the natural sign-change action of
\((C_2)^{\ell}\) on the Stiefel manifold \(V_{n, \ell}\).

For \(i=0,\dots,\ell\) let
\(
m_i=(1,\dots,1,-1,1,\dots,1)\in (C_2)^{\ell+1}
\)
denote the standard generators, where the entry \(-1\) appears in the \(i+1\)-st position. The action of \((C_2)^{\ell+1}\) on 
\(V
\)
is described as follows. For \((x_0,x_1,\dots,x_\ell)\in V\) and \(i=1,\dots,\ell\),  \(m_i\) changes the sign of every coordinate of \(x_i\) and fixes all other components.
Finally, 
\(m_0
\) changes the sign of every coordinate of $x_0$ and the first coordinate of each of $x_1,\cdots, x_{\ell}$.

Note that, as a $(C_{2})^{\ell+1}$-representation, $V$ decomposes as
\begin{equation}\label{decom}
    V \;\cong\; \chi_{0}^{\oplus n}
\;\oplus\;
\bigoplus_{i=1}^{\ell}
\Bigl(
(\chi_{0}\otimes \chi_{i})
\;\oplus\;
\chi_{i}^{\oplus (n-i-1)}
\Bigr),
\end{equation}
where the decomposition is taken with respect to the irreducible characters
$\chi_i$ defined in \S\ref{not}.

\subsection{An $RO(G)$-graded spectral sequence}

We now describe an $RO(G)$-graded Bredon cohomological spectral sequence that helps up to understand the  \(\tH^\bs_G(S^{0};\uFz)\)-module structure on the 
\(RO(G)\)-graded Bredon cohomology of both the universal space \(EG\) 
and the product space \(S^{n}\times V_{n, \ell}\).

\medskip
\noindent
Before we start, note that for any \(G\)-space \(X\), the canonical collapsing map
\(
   X_{+} \to S^{0}
\)
is \(G\)-equivariant, and therefore induces a 
\(\tH^\bs_G(S^{0};\uFz)\)-module structure on 
\(\tH^\bs_G(X_{+};\uFz)\).

\medskip
\noindent
Consider a \(G\)-CW complex decomposition
\(
   X = \bigcup_{s} X^{(s)}.
\)
If \(X\) is free, each \(s\)-cell is of the form \(G/e\times \mathbb{D}^{s}\).
In this case, the orbit space \(X/G\) inherits a CW structure whose 
\(s\)-skeleton is \(X^{(s)}/G\), and we have
\[
   X^{(s)}/X^{(s-1)} \;\cong\; \bigvee_{e\in I(s)} G_{+}\wedge S^{s}.
\]

\medskip
\noindent
Following \cite{BG21}, one obtains an \(RO(G)\)-graded spectral sequence:

\begin{prop}\label{HFPSS}
For a free $G$-CW complex $X$, there exists a spectral sequence
\[
   E^{s,t}_{2}(\alpha)
   = H^{s}\!\left(X/G; \pi_{t}\!\left(S^{-\dim(\alpha)}\wedge H\F\right)\right)
   \;\Longrightarrow\;
   H^{s-t-\alpha}_{G}(X;\uFz),
\]
with differentials
\[
   d_{r}\colon
   E^{s,t}_{r}(\alpha)
   \longrightarrow
   E^{s+r,t-r+1}_{r}(\alpha).
\]
These spectral sequences assemble (as \(\alpha\) varies) into a multiplicative
\(RO(G)\)-graded spectral sequence
\[
   E^{s,\alpha}_{2}
   =
   H^{s}\!\left(X/G;\pi_{0}\!\left(S^{-\dim(\alpha)}\wedge H\F\right)\right)
   \;\Longrightarrow\;
   H^{s-\alpha}_{G}(X;\uFz),
\]
where \(s\in\Z\) and \(\alpha\in RO(G)\).
\end{prop}

\medskip
\noindent
The  $E_2$-page is concentrated
entirely along the single horizontal line
\(
t = -\dim(\alpha).
\) 
For $r \ge 2$, if $E^{s,t}_r(\alpha)$ is nonzero then necessarily
$t = -\dim(\alpha)$, while the target bidegree satisfies
$t - r + 1 \neq -\dim(\alpha)$. Hence the target group is zero, and therefore
\(
d_r = 0  \text{ for all } r \ge 2.
\)
Thus the spectral sequence collapses at the $E_2$-page.

The multiplicative structure in Proposition~\ref{HFPSS} shows that 
every element can be written as a product of the form
\[
   H^{s}(X/G) \otimes 
   \pi_{0}\!\left(S^{-\dim\alpha}\wedge H\F\right).
\]

The groups $\pi_0\!\left(S^{-\dim(\alpha)} \wedge H\F\right)$ assemble, as $\alpha$
varies, into a graded ring of the form
\[
\bigotimes_{\chi\in \Hom(G, C_2)} \F\bigl[u_{\chi}^{\pm}\bigr],
\] hence we have:
\begin{cor}\label{uni}
The Bredon cohomology of \(EG\) is
\[
   \tH^{\bs}_{G}\!\left(E G_{+};\uFz\right)
   \;\cong\;
   \F\!\left[
      x_{i},\,u_{\chi}^{\pm}
      \;\middle|\;
      \chi\in G^\circ ,\;1 \le i\le\ell
   \right],
\]
with \(|x_{i}|=1\).
\end{cor}

\begin{remark}
We identify the class \(a_{\chi_i}\) with \(x_i\,u_{\chi_i}\).
For further details, see \cite[page~7]{BG21}.
\end{remark}

\medskip
\noindent
Since \(G\) acts freely on \(S^{n}\times V_{n, \ell}\),
the same method used in Corollary~\ref{uni} yields:

\begin{cor}\label{cbrecoh}
The \(RO(G)\)-graded Bredon cohomology of the ordered configuration space 
\(S^{n}\times V_{n, \ell}\) is
\[
   \tH^{\bigstar}_{G}\!\left(S^{n}\times V_{n, \ell};\uFz\right)
   \;\cong\;
   H^{\ast}\!\left(\RR P^{n}\times 
     G_{\RR}(1,\dots,1,n-\ell);\F\right)
   \;\otimes\;
   \bigotimes_{\chi\in G^\circ} \F\bigl[u_{\chi}^{\pm}\bigr].
\]
\end{cor}

\section{Cohomology of certain quotient of $O(n)$}
In this section, we investigate the cohomology of certain quotient spaces that naturally arise in the study of Stiefel and flag manifolds.  
We begin by recalling classical computations of the cohomology rings of Stiefel manifolds, orthogonal groups, and their classifying spaces.  

\begin{prop}[\cite{Bor53}]
Let $1 \leq \ell \leq n$. Then:
\begin{enumerate}
    \item \( H^\ast (V_{n,\ell}; \F) \cong \F[z_{n-\ell}, \ldots, z_{n-1}] \).
    \item \( H^\ast (O(n); \F) \cong \F[y_1, \ldots, y_{n-1}] \).
    \item \( H^\ast (BO(n); \F) \cong \F[\omega_1, \ldots, \omega_n] \).
\end{enumerate}
\end{prop}
\subsection{Flag Manifolds over \(\mathbb{R}\)}

Let \(n_{1},\dots,n_{s}\) be positive integers with \(n_{1}+\cdots+n_{s}=n\).  
An \emph{\((n_{1},\dots,n_{s})\)-flag over \(\mathbb{R}\)} is an ordered collection \(\sigma=(\sigma_{1},\dots,\sigma_{s})\) of mutually orthogonal subspaces of \(\mathbb{R}^{n}\), where \(\dim\sigma_{i}=n_{i}\).  
The space of all such flags is a compact smooth manifold called the \emph{real flag manifold}, denoted \(G_{\mathbb{R}}(n_{1},\dots,n_{s})\).

\begin{remark}
The case \(G_{\mathbb{R}}(n_{1},n_{2})\) is the Grassmannian of \(n_{1}\)-planes in \(\mathbb{R}^{n_{1}+n_{2}}\), i.e.~\(G_{\mathbb{R}}(n_{1},n_{2}) \cong \mathrm{Gr}_{n_{1}}(\mathbb{R}^{n_{1}+n_{2}})\).
\end{remark}
The real flag manifold can be described as the homogeneous space 
\[
G_{\mathbb{R}}(n_{1}, n_{2}, \dots, n_{s}) \;\cong\; \frac{O(n)}{O(n_{1}) \times O(n_{2}) \times \cdots \times O(n_{s})}.
\] The group \(O(n)\) acts transitively on flags, with stabilizer the block-diagonal subgroup \(O(n_{1}) \times \cdots \times O(n_{s})\), giving \(G_{\mathbb{R}}(n_{1},\dots,n_{s})\) its natural smooth structure and an \(O(n)\)-invariant Riemannian metric.

Note that $V_{n,\ell}\cong \dfrac{O(n)}{O(n-\ell)}$ and the group $(C_2)^\ell=O(1)^\ell$ acts freely on it as described in \S 3.1. The quotient map
\[
V_{n,\ell} \longrightarrow G_{\mathbb{R}}(1,\ldots,1,n-\ell)
\]
is therefore a principal $(C_2)^\ell$--bundle. Consequently, it is classified
by a map to the classifying space $B(C_2)^\ell \simeq (\mathbb{RP}^\infty)^\ell$,
and we obtain a homotopy fibration
\begin{equation}\label{eq:flag-fibration}
V_{n,\ell} \;\longrightarrow\; G_{\mathbb{R}}(1, \ldots, 1, n-\ell) \;\longrightarrow\; (\mathbb{RP}^\infty)^{\ell},
\end{equation}

\subsection{The spectral sequence computation}

We apply the Serre spectral sequence in cohomology (with coefficients in a field \(\F\)) to the fibration \eqref{eq:flag-fibration}.  
The \(E_2\)-page has the form
\[
E_2^{p,q} \;=\; H^p\!\left( (\mathbb{RP}^\infty)^{\ell}; \, \mathcal{H}^q(V_{n,\ell}; \F) \right)  
\quad\Longrightarrow\quad H^{p+q}\!\left( G_{\mathbb{R}}(1, \ldots, 1, n-\ell); \F \right),
\]
where \(\mathcal{H}^q(V_{n,\ell}; \F)\) denotes the local coefficient system on the base induced by the monodromy action of \(\pi_1\big((\mathbb{RP}^\infty)^{\ell}\big)\) on \(H^q(V_{n,\ell}; \F)\).

The fundamental group \(
\pi_1\big((\mathbb{RP}^\infty)^{\ell}\big) \;\cong\; (\mathbb{Z}/2)^{\ell}
\) acts on \(V_{n,\ell}\) by sign changes in each coordinate:
\[
(\lambda_1, \ldots, \lambda_\ell) \cdot (v_1, \ldots, v_\ell) 
\;=\; (v_1, \ldots, v_\ell) \cdot \operatorname{diag}(\lambda_1, \ldots, \lambda_\ell),
\]
where each \(\lambda_i \in \{\pm 1\}\). Over \(\F\)-coefficients, this action is trivial on cohomology.  
Therefore the local system \(\mathcal{H}^q(V_{n,\ell}; \F)\) is constant, and the \(E_2\)-page splits as a tensor product:
\[
E_2^{p,q} \;\cong\; H^p\!\left( (\mathbb{RP}^\infty)^\ell; \F \right) \otimes_{\F} H^q(V_{n,\ell}; \F).
\]

Note that the cohomology of \((\mathbb{RP}^\infty)^{\ell}\) is a polynomial algebra
\[
H^*\!\left( (\mathbb{RP}^\infty)^\ell; \F \right) \;\cong\; \F[x_1, \ldots, x_\ell],
\]
where \(x_i = w_1(L_i) \in H^1(\mathbb{RP}^\infty; \F)\) is the first Stiefel–Whitney class of the tautological line bundle \(L_i\) over the \(i\)-th factor.

The cohomology of \(V_{n,\ell}\) is generated by classes
\(
z_{n-\ell}, z_{n-\ell+1}, \ldots, z_{n-1}
\)
in degrees \(n-\ell, n-\ell+1, \ldots, n-1\), respectively. Thus
\[
E_2^{p,q} \;\cong\; \F[x_1, \ldots, x_\ell] \otimes_{\F} \F[z_{n-\ell}, \ldots, z_{n-1}].
\]

Consider the following homotopy commutative diagram of fibrations
\(({\bf A}), ({\bf B}), ({\bf C})\), and \(({\bf D})\):
\[
\xymatrix{
({\bf A}) & ({\bf B}) & ({\bf C}) & ({\bf D})\\
V_{n,\ell} \ar[d] \ar@{=}[r]
  & V_{n,\ell} \ar[d]
  & O(n) \ar[d] \ar[l]^{\phi} \ar@{=}[r]
  & O(n)\ar[d] \\
G_{\mathbb{R}}(1,\ldots,1,n-\ell) \ar[d] \ar[r]
  & \dfrac{O(n)}{O(\ell)\times O(n-\ell)} \ar[d]
  & \dfrac{O(n)}{O(\ell)\times O(n-\ell)} \ar[d] \ar@{=}[l] \ar[r]
  & EO(n)\ar[d] \\
BO(1)^\ell \ar[r]_-{g}
  & BO(\ell)
  & BO(\ell)\times BO(n-\ell) \ar[l]^-{\pi_1} \ar[r]_-{f}
  & BO(n).
}
\]

We analyze the induced Serre spectral sequences and their
differentials using naturality.

\medskip

Lets begin with the universal fibration \(({\bf D})\),
\[
O(n) \longrightarrow EO(n) \longrightarrow BO(n).
\]
It is classical that the transgression in the Serre spectral sequence satisfies
\[
\tau(y_i)=w_{i+1},
\]
where \(w_{i+1}\) is the \((i+1)\)-st Stiefel--Whitney class of the universal
\(n\)-plane bundle over \(BO(n)\).

\medskip

Next, consider the fibration \(({\bf C})\).
The map
\[
f \colon BO(\ell)\times BO(n-\ell) \longrightarrow BO(n)
\]
classifies the  vector bundles
\(
\xi_\ell \times \xi_{n-\ell},
\)
where \(\xi_k\) denotes the canonical \(k\)-plane bundle over \(BO(k)\).
By naturality of the Serre spectral sequence with respect to the map \(f\),
the differential in \(({\bf C})\) is given by
\[
d_{i+1}(y_i)
=
w_{i+1}(\xi_\ell \times \xi_{n-\ell})
\in H^{i+1}(BO(\ell)\times BO(n-\ell);\mathbb{F}_2).
\]
Let $w_i = w_i(\xi_{\ell})$, $\tilde w_i = w_i(\xi_{n-\ell})$, and write
\[
w = 1 + w_1 + \cdots + w_{\ell}, \qquad 
w' = 1 + w'_1 + \cdots
\]
for the inverse of $w$. By Proposition~11.1 of~\cite{Bor53a}, we have
\[
\sum_{i+j=r} \tilde w_i w_j = 0 \qquad \text{for all } r \le n-\ell.
\]
Comparing this with
\[
(1+w_1+\cdots+w_{\ell})(1+w'_1+\cdots)=1
\]
shows that $\tilde w_i = w'_i$ for all $i \le n-\ell$. Applying this on the
$E_{n-\ell+1}$-page gives
\[
\begin{aligned}
d_{n-\ell+1}(y_{n-\ell})
&= w_{n-\ell+1}(\xi_{\ell}\times\xi_{n-\ell}) \\
&= \sum_{\substack{i+j=n-\ell+1\\ i\le \ell,\; j\le n-\ell}} \tilde w_i w_j
= \sum_{\substack{i+j=n-\ell+1\\ i\le \ell,\; j\le n-\ell}} w_i w'_j
= -\,w'_{n-\ell+1}.
\end{aligned}
\]
The same argument yields $d_{j+1}(y_{j})=-w'_{j+1}$ for all $j\ge n-\ell$.
\medskip

Now consider the fibration \(({\bf B})\).
The map
\[
\phi \colon O(n) \longrightarrow V_{n,\ell}
\]
induces a pullback on cohomology. A direct inspection of the generators shows
that
\[
y_1,\ldots,y_{n-\ell-1} \notin \operatorname{Im}(\phi^*),
\]
while for \(j \ge n-\ell\) we have
\(
z_j=\phi^*(y_j).
\)

 Comparing the Serre spectral sequences for \(({\bf B})\) and \(({\bf C})\) yields
\[
d_{j+1}(z_j)
=
-w'_{j+1}
\qquad j \ge n-\ell.
\]

\medskip

This completely determines the nontrivial differentials in the Serre spectral
sequence associated to the fibration \(({\bf B})\).

We now determine the differentials in the Serre spectral sequence associated to
the fibration \(({\bf A})\),
\[
V_{n,\ell}
\longrightarrow
G_{\mathbb{R}}(1,\ldots,1,n-\ell)
\longrightarrow
BO(1)^\ell .
\]

Note that $g \colon BO(1)^\ell \to BO(\ell)$ is the classifying map for the Whitney
sum $\gamma_1 \oplus \cdots \oplus \gamma_\ell$. Hence $g^\ast(w')= \prod_{i=1}^\ell (1+x_i)^{-1}.$

Consequently, for $j=n-\ell, \cdots, n-1$, the differential (for the spectral sequence \(({\bf A})\)) is given by
\[
d_{j+1}(z_j)
= -g^\ast(w'_{j+1})
=
f_{j+1}(x_1,\ldots,x_\ell) \mbox{ (mod }2),
\]
where $f_{j+1}$ denotes the complete symmetric polynomial of degree $j+1$.

All other differentials vanish for degree reasons. These relations completely determine the \(E_\infty\)-page and hence we have 

\begin{prop}\label{ppartialflag}
The cohomology ring of the real flag manifold satisfies
\[
H^\ast \left( G_{\mathbb{R}}(1, \ldots, 1, n-\ell); \F \right) \cong \F[x_1, \ldots, x_\ell] \, / \, (f_{n-\ell+1}, \ldots, f_{n}),
\]
where $f_j$ denotes the complete symmetric polynomial of degree $j$ in $x_1, \ldots, x_\ell$.
\end{prop}

\begin{remark}\label{rem}
The top nontrivial cohomology class of $G_{\mathbb{R}}(1, \cdots, 1, n -\ell)$ is
\[\prod_{i=1}^\ell x_i^{n-i},
\]
which generates the one-dimensional top cohomology group.
\end{remark}

\section{The Proof}
We begin by recalling some key structural facts about representation-graded Bredon cohomology rings due to Hausmann-Schwede \cite{HS24}.  These results will provide the algebraic foundation for all subsequent arguments.

Let $\widetilde{H}^{\Rep}_{G}(S^0)$ denote the
$RO(G)$-graded reduced Bredon cohomology of $S^0$ consisting of those
degrees of the form $V-k$, where $V$ is a finite-dimensional real
$G$-representation and $k\in\mathbb Z$. Note that in \cite{HS24} it is denoted by $H(G, \bs).$ In our notation $H_m(G, V)$ corresponds to $\widetilde{H}^{V-m}_G(S^0;\underline{\mathbb{F}}_2)$.

\begin{prop}[{\cite[Proposition 2.1]{HS24}}]\label{agen}
Let $G$ be an elementary abelian $2$-group and let $V$ be a $G$-representation
with trivial fixed points.

\begin{enumerate}
\item
For every subgroup $H \leq G$, the restriction homomorphism
\[
\res^G_H \colon
\widetilde{H}^{V-m}_G(S^0;\underline{\mathbb{F}}_2)
\longrightarrow
\widetilde{H}^{V-m}_G(G/H_+;\underline{\mathbb{F}}_2)
\]
is surjective.

\item
Suppose 
$\chi$ is a nontrivial $G$-character with kernel $K$. Then the following sequence is exact:
\[
0 \to
\widetilde{H}^{\,W-m}_G(S^0;\underline{\mathbb{F}}_2)
\stackrel{\cdot a_\chi}{\to}
\widetilde{H}^{\,W\oplus \chi -m}_G(S^0;\underline{\mathbb{F}}_2)
\stackrel{\res^G_K}{\to}
\widetilde{H}^{W\oplus \chi -m}_G(G/K_+;\underline{\mathbb{F}}_2)
\to 0 .
\]

\item
The $\mathbb{F}_2$-vector space
\(\widetilde{H}^{\,V-m}_G(S^0;\underline{\mathbb{F}}_2)
\)
is spanned by the classes
\[
a_U \cdot u_W
\]
for all $G$-representations $U$ and $W$ such that $U \oplus W = V$ and
$m = \dim(W)$.
\end{enumerate}
\end{prop}
\begin{theorem}[{\cite[Theorem 2.2]{HS24}}]\label{intdom}
For every elementary abelian $2$-group \(G\), the representation-graded Bredon homology ring
\(
\tilde{H}^\Rep_G(S^0; \uFz)
\)
is an integral domain.
\end{theorem}

\begin{prop}\label{psphere}
Let \(G\) be an elementary abelian 2-group and \(V\) a \(G\)-representation with trivial fixed points. Then
\[
\tilde{H}^V_G(S(V)_+; \uFz) \cong 0.
\]
\end{prop}

\begin{proof}
Consider the cofibration sequence of based \(G\)-spaces
\[
S(V)_+ \xrightarrow{i} S^0 \xrightarrow{j} S^V,
\]
where \(S(V)_+\) denotes the unit sphere \(S(V)\) with a disjoint basepoint.

Applying the reduced representation-graded Bredon cohomology functor \(\tilde{H}^V_G(-; \uFz)\), we get the long exact sequence
\[
\cdots \to \tilde{H}^V_G(S^V; \uFz) \xrightarrow{j^\bs} \tilde{H}^V_G(S^0; \uFz) \xrightarrow{i^\bs} \tilde{H}^V_G(S(V)_+; \uFz) \to \tilde{H}^{V+1}_G(S^V; \uFz) \to \cdots.
\]
Since  \(a_V\) is not a zero divisor (as \(\tilde{H}^\Rep_G(S^0; \uFz)\) is a domain by Theorem \ref{intdom}) and the map \(j^\bs\) corresponds to multiplication by \(a_V\), thus, \(j^\bs\) is injective, and thus \(i^\bs\) is the zero since $\tilde{H}^V_G(S^0; \uFz)\cong \F$ by Proposition \ref{agen} (see part (3) for $m=0$). Consequently,
\[
\tilde{H}^V_G(S(V)_+; \uFz) \cong 0,
\]
as claimed.
\end{proof}

We now turn to the description of the $\tilde{H}^\bs_G(S^0; \uFz)$–module structure in the case $S^{n}\times V_{n, \ell},$
with $G=(C_2)^{\ell+1}$. 
\begin{prop}\label{action-2}
Let \(G=(C_{2})^{\ell+1}\) and \(X = S^{n} \times V_{n, \ell}\).  
Then \(\tilde{H}^\bigstar_G(X_+;\uFz)\) is an $\tilde{H}^\bs_G(S^0; \uFz)$-module with action:
\begin{enumerate}
    \item For each nontrivial \(\chi \in G^\circ\), the class \(u_\chi \in \tilde{H}^\bs_G(S^0; \uFz)\) acts by multiplication with the corresponding unit \(u_\chi\).
    \item The Euler class \(a_\chi\) acts by multiplication with \(x_\chi u_\chi\), where \(x_\chi \in H^1(BG;\F)\) is the degree-one class associated to \(\chi\).
\end{enumerate}
\end{prop}

\begin{proof}
Since \(G\) acts freely on \(X\), there exists (up to \(G\)-homotopy) a classifying map \(\phi\colon X \to EG\). The collapse map factors as  
\(X_{+}\xrightarrow{\;\phi_{+}\;} EG_{+}\xrightarrow{\;q\;} S^{0}\),  
so the $\tilde{H}^\bs_G(S^0; \uFz)$-action on \(\tilde H_G^\bs(X_+;\uFz)\) is determined by the pullback \(\phi_+^\bs\).

By Corollary~\ref{uni}, \(\tilde H_G^\bs(EG_+;\uFz)\) is generated by invertible Thom classes \(u_\chi\) and Euler classes \(a_\chi = x_\chi u_\chi\). Pulling back gives \(\phi_+^\bs(u_\chi) = u_\chi\). For Euler classes we obtain \(\phi_+^\bs(a_\chi) = \phi_+^\bs(x_\chi)\,u_\chi\), where \(\phi_+^\bs(x_\chi)\) is the natural degree-one class on \(X/G\): for \(\chi=\chi_0\) it is the generator \(x_0\in H^1(\RR P^n;\F)\), and for \(\chi=\chi_i\) with \(1\le i\le \ell\) it is \(x_i\) from the partial flag factor as in Proposition \ref{ppartialflag}. This yields the stated formulas.
\end{proof}

\begin{cor}
Let $V$ be as in \eqref{decom}. Then the action of $a_V$ on
\[
\widetilde{H}_G^{\bs}\!\left(S^{n} \times V_{n, \ell}; \underline{\mathbb{F}}_{2} \right)
\]
is given by multiplication by the element
\[
x_{0}^{n}\,u_{0}^{n}
\cdot
\prod_{i=1}^{\ell}
\,x_{i}^{\,n-i}
\cdot
u_{\chi_{0}\otimes \chi_{i}}\,
u_{\chi_{i}}^{\,n-i-1}.
\]
\end{cor}

\begin{proof}
The class $a_V$ is computed multiplicatively from the Euler classes of the summands in the decomposition of $V$. More precisely,
\[
a_V
=
a_{\chi_0}^{\,n}
\prod_{i=1}^{\ell}
a_{\chi_0\otimes \chi_i}\,
a_{\chi_i}^{\,n-i-1},
\]
where $a_{\chi_i}$ denotes the Euler class of the real line bundle associated with the character $\chi_i$.

Since the first Stiefel--Whitney class is additive under tensor products of line bundles, we have
\[
a_{\chi_0\otimes \chi_i}
=
a_{\chi_0}+a_{\chi_i}.
\]
The claim now follows directly from Proposition~\ref{action-2} and the fact that $x_0^{n+1}=0$.
\end{proof}

\begin{theorem}
There does not exist any \((C_{2})^{\ell+1}\)-equivariant map
\[
   S^{n}\times V_{n, \ell} \longrightarrow S(V),
\]
with \(V\) as described in \eqref{decom}.
\end{theorem}

\begin{proof}
Consider $G=(C_2)^{\ell+1}$. Assume for contradiction that such a map \(f\) exists.
Using the \(\tH_G^{\bs}(S^{0};\uFz)\)-module structure, 
we obtain the following commutative diagram:
\[
\xymatrix{
\tH_G^{0}\!\left(S(V);\uFz\right) \ar[r]^{f^\ast} \ar[d]_{a_V.} &
\tH_G^{0}\!\left(S^{n}\times V_{\ell}(\RR^{d});\uFz\right) \ar[d]_{a_V.} \\
\tH_G^{V}\!\left(S(V)_{+};\uFz\right) \ar[r]^{f^\bs} &
\tH_G^{V}\!\left(S^{n}\times V_{\ell}(\RR^{d});\uFz\right)
}
\]

By Proposition \ref{psphere}, we yield:
\[
0=f^\bs(a_V.1)=a_V.f^\ast(1)=x_{0}^{n}\,u_{0}^{n}
\cdot
\prod_{i=1}^{\ell}
\,x_{i}^{\,n-i}
\cdot
u_{\chi_{0}\otimes \chi_{i}}\,
u_{\chi_{i}}^{\,n-i-1}.
\]
which is non-trivial by Corollary \ref{cbrecoh} and Remark \ref{rem}. This contradicts our assumption, hence the result follows.

\end{proof}

\end{document}